\documentclass{amsart}
\usepackage{amssymb}
\usepackage{graphicx,overpic} % Required for inserting images
\usepackage{color}

\usepackage[hidelinks,pdftex]{hyperref}

\AtBeginDocument{%
   \def\MR#1{}
}

\numberwithin{equation}{section}
\theoremstyle{plain}
\newtheorem{theorem}[equation]{Theorem}
\newtheorem{corollary}[equation]{Corollary}
\newtheorem{lemma}[equation]{Lemma}

\newtheorem*{namedtheorem}{\theoremname}
\newcommand{\theoremname}{testing}
\newenvironment{named}[1]{\renewcommand{\theoremname}{#1}\begin{namedtheorem}}{\end{namedtheorem}}

\theoremstyle{definition}
\newtheorem{definition}[equation]{Definition}
\newtheorem{remark}[equation]{Remark}
\newtheorem{algorithm}[equation]{Algorithm}
\newtheorem{example}[equation]{Example}

\newcommand{\HH}{{\mathbb{H}}}
\newcommand{\RR}{{\mathbb{R}}}
\newcommand{\ZZ}{{\mathbb{Z}}}
\newcommand{\NN}{{\mathbb{N}}}
\newcommand{\CC}{{\mathbb{C}}}
\newcommand{\QQ}{{\mathbb{Q}}}

\renewcommand{\setminus}{{\smallsetminus}}

\DeclareMathOperator{\PSL}{PSL}

\DeclareMathOperator{\Vol}{Vol}
\DeclareMathOperator{\UT}{UT}
\DeclareMathOperator{\PT}{PT}
\newcommand{\SigmaMod}{{\Sigma_{\rm{Mod}}}}
\DeclareMathOperator{\Trace}{Trace}
\title{Arithmetic modular links}

\author{Tali Pinsky}
\address[]{Department of Mathematics, Technion, Haifa, 32000 Israel} 
\email[]{talipi@technion.ac.il} 

\author{Jessica S.~Purcell}
\address[]{School of Mathematics, Monash University, Clayton, VIC 3800, Australia } 
\email[]{jessica.purcell@monash.edu} 

\author{Jos\'e Andr\'es Rodr\'iguez-Migueles}
\address[]{Mathematisches Institut der Universit\"at M\"unchen, Theresienstr. 39, 80333, M\"unchen, Germany } 
\email[]{migueles@math.lmu.de} 

\date{\today}

\begin{document}

\begin{abstract}
We construct a sequence of geodesics on the modular surface such that the complement of the canonical lifts to the unit tangent bundle are arithmetic 3-manifolds. 
\end{abstract}

\maketitle

\section{Introduction}
The modular group $\PSL(2,\ZZ)$ is one of the simplest examples of an arithmetic group. 
The quotient of the upper half plane by the modular group is called the modular surface $\SigmaMod$; it is an arithmetic hyperbolic 2-dimensional orbifold. 

One dimension higher, arithmetic hyperbolic 3-manifolds and 3-orbifolds form families of manifolds with very rich structure. They are also quite special. For example, among knot complements, only the figure-8 knot is arithmetic~\cite{Reid:ArithmeticKnot}, and there exist closed orientable 3-manifolds that do not contain a simple closed curve with arithmetic complement~\cite{BakerReid:Arithmetic}. However, every closed orientable 3-manifold contains an arithmetic link~\cite{HildenLozanoMontesinos}.

Associated with each oriented closed geodesic $\gamma$ on the modular surface is a 3-manifold. This is obtained by lifting the geodesic $\gamma$ into the unit tangent bundle over the modular surface $\UT(\SigmaMod)$ to obtain a corresponding periodic orbit of the geodesic flow $\widehat{\gamma}$ called the \emph{canonical lift}. The 3-manifold is the complement of $\widehat{\gamma}$ in the unit tangent bundle. 

By Thurston's hyperbolization theorem, the complement of a canonical lift of a closed modular geodesic will always be hyperbolic; see Foulon and Hasselblatt~\cite{FoulonHasselblatt}.
What is unknown in general is whether it will be arithmetic for some cases, and if so, what topological, geometric, and algebraic properties of the geodesic yield arithmeticity. 
 
In this paper, we find an explicit family of canonical lift complements that are arithmetic. 

\begin{theorem}\label{Thm:MainArithmetic}
There exists a sequence $\{\gamma_n\}_{n\in\NN}$ of distinct closed geodesics on the modular surface such that for each $n$, the union of the first $n$ canonical lifts $\bigcup_{j=1}^n \widehat{\gamma_j}$ has complement in the unit tangent bundle $\UT(\SigmaMod)$ that is an arithmetic hyperbolic 3-manifold $\UT(\SigmaMod)\setminus\bigcup_{j=1}^n\widehat{\gamma_j}$, obtained by gluing regular ideal octahedra. 
\end{theorem}

Note that for $n>1$, the manifolds of Theorem~\ref{Thm:MainArithmetic} are complements of more than one geodesic. When $n=1$, the theorem produces a 3-manifold homeomorphic to the Whitehead link complement, which is well known to be arithmetic~\cite[\S~4.5]{MaclachlanReid}. 
This corresponds to $\UT(\SigmaMod)\setminus\widehat{\gamma_0}$ for $\gamma_0$ the shortest geodesic on the modular surface.
It is an open question as to whether this is the only arithmetic canonical lift complement of a single geodesic on the modular surface. 

The theorem is proved by considering canonical lifts of geodesics on a once punctured torus, which is a six-fold cover of the modular surface. In Theorem~\ref{Thm:FareyPath} below, we build an explicit family of geodesics on the punctured torus and we prove that their canonical lifts are built of regular ideal octahedra. Such manifolds are always arithmetic, and the main theorem follows as arithmeticity is invariant under finite covers. 

Because of the explicit nature of the construction, we are further able to obtain geometric information on these manifolds. For example, their volumes are given explicitly, and can be related to the lengths of the geodesics. 

\begin{corollary}\label{Cor:ModularVolumes}
There exists a sequence $\{\gamma_k\}_{k\in\mathbb{N}}$ of closed geodesics on the modular surface with length $\ell({\gamma_k})\nearrow \infty$ such that for $\Gamma_n:=\bigcup_{k=1}^n \gamma_k$
\begin{enumerate}
  \item $\UT(\SigmaMod)\setminus{\widehat{\Gamma}_n}$ is arithmetic,
  \item $\Vol(\UT(\SigmaMod)\setminus{\widehat{{\Gamma}_n}})= n\,v_{oct}/2$, and
  \item $\Vol(\UT(\SigmaMod)\setminus{\widehat{{\Gamma}_n}})\asymp \sqrt{\ell({\Gamma_n})}$.
\end{enumerate}
Here $v_{oct}$ is the volume of a regular ideal octahedron. 
\end{corollary}

In Corollary~\ref{Cor:ModularVolumes}, $\asymp$ means \emph{coarsely equivalent}: there are constants $A$, $B$, $C$, and $D$ such that 
\[A\sqrt{\ell({\Gamma_n})}+B \leq \Vol(\UT(\SigmaMod)\setminus{\widehat{{\Gamma}_n}}) \leq C\sqrt{\ell({\Gamma_n})} + D.\]

Note that others have related volume to length of geodesics. Bergeron, Pinsky, and Silberman showed that the volume is bounded by a constant times the length~\cite{BergeronPinskySilberman}. Rodriguez-Migueles showed that there is a sequence of geodesics such that the volume grows linearly 
in the length of the geodesics up to a logarithmic factor~\cite{Rodriguez:Lowerbound}. Upper and lower bounds were extended by Cremaschi and Rodriguez-Migueles~\cite{CremaschiRM}. Cremaschi, Rodriguez-Migueles and Yarmola related volumes of the canonical lifts of a pair of simple closed curves to the Weil-Petersson distance in Teichm\"uller space~~\cite{CRMY}. 

More generally, by taking finite covers, we obtain:

\begin{corollary}\label{Cor:Cover}
Let $\Sigma_{g,r}$ be an orientable punctured surface with any hyperbolic metric. Then there exists a sequence $\{{\Gamma_k}\}_{k\in\mathbb{N}}$ of filling finite sets of closed geodesics on $\Sigma_{g,r}$ with lengths $\ell({\Gamma_k})\nearrow \infty$, such that  $\UT(\Sigma_{g,r})\setminus{\widehat{{\Gamma_k}}}$ is arithmetic for each $k\in\mathbb{N}$ and
\[ 
\Vol(\UT(\Sigma_{g,r})\setminus{\widehat{{\Gamma_k}}})\asymp \sqrt{\ell({\Gamma_k})}.
\]
\end{corollary}

\subsection{Acknowledgements}
The authors thank the anonymous referee for their helpful suggestions that have improved the paper. They also thank Neil Hoffman for pointing them towards helpful references. 
This work began during a workshop at the Matrix Centre in Australia; the authors thank the organisers and the centre for hosting them. Purcell was partially supported by the Australian Research Council, grant DP210103136. Rodr\'iguez-Migueles is also very grateful for discussions with Juan Souto and Connie On Yu Hui on related topics, and recognises the support by the Special Priority Programme SPP 2026 Geometry at Infinity funded by the DFG.

\section{Surfaces and unit tangent bundles}\label{Sec:Surfaces}
Let $\Sigma$ be a hyperbolic surface or orbifold. The unit tangent bundle $\UT(\Sigma)$ consists of points of the form $(x,v)$, where $x$ lies on $\Sigma$, and $v$ is a unit vector tangent to $\Sigma$ at $x$. Given a smooth oriented curve $\gamma$ on $\Sigma$, any point $x\in \gamma$ determines a point $(x,v)$ in the unit tangent vector, by letting $v$ be the unit vector at $x$ pointing in the direction of $\gamma$. Then $\gamma$ lifts to a embedded closed curve $\widehat\gamma$ in $\UT(\Sigma)$. 

\subsection{The modular surface}
The \emph{modular surface} is the quotient of $\HH^2$ by the modular group $\PSL(2,\ZZ)$. Background on the modular group can be found in many places, for example in work of Series~\cite{Series:ModularSurface}; see also Brandts, Pinsky, and Silberman~\cite{BPS}. We review a few relevant facts here. 

Consider the upper half plane $\HH^2$ with its hyperbolic metric. Let $U$ be a rotation of $\pi$ about the point $i$ and let $V$ be a rotation of $2\pi/3$ about the point $\frac{1}{2}+i\frac{\sqrt{3}}{2}$, permuting points $\infty, 1, 0$.  These two rotations generate the modular group $\PSL(2,\ZZ)$.
As elements of $\PSL(2,\ZZ)$, $U$ and $V$ have the form
 \[ U = \pm \begin{pmatrix} 0 & -1 \\ 1 & 0 \end{pmatrix} \quad 
 V = \pm \begin{pmatrix} 0 & -1 \\ 1 & -1 \end{pmatrix} \]
The rotation $V$ fixes the hyperbolic ideal triangle in $\HH^2$ with vertices $0, 1, \infty$, while $U$ maps it to an adjacent ideal triangle. Thus the orbit of this ideal triangle under $\PSL(2,\ZZ)$ is an invariant tessellation of $\HH^2$ by ideal triangles called the Farey tessellation. It has an ideal vertex at each point of $\mathbb{Q}\cup\infty$ on $\partial\HH^2$. 

The quotient of $\HH^2$ by the modular group $\PSL(2,\ZZ)$ is an orbifold that is a sphere with a cusp, a cone point of order three, and a cone point of order two. This is called the \emph{modular surface} and denoted $\SigmaMod$. A fundamental domain for $\SigmaMod$ is given by one third of the $0, 1, \infty$ ideal triangle.

Elements of finite order in $\PSL(2,\ZZ)$ are exactly the conjugates of $1, U, V, V^2$.
Every element of infinite order {is a finite word in $U$, $V$ and $V^{-1}=V^2$, involving both letters. Conjugating, one may always obtain a word beginning with $V$ or $V^2$ and ending with $U$. Thus, up to conjugation, any infinite order element can be written in positive powers of
$L=V^2U$ and $R=VU$} \cite{Ghys}, where 

\begin{equation}\label{Eqn:Generators}
 L= \pm \begin{pmatrix} 1 & 1 \\ 0 & 1 \end{pmatrix} \quad \mbox{and} \quad
R= \pm \begin{pmatrix} 1 & 0 \\ 1 & 1 \end{pmatrix}.
\end{equation}

A closed geodesic on the modular surface $\SigmaMod$ is called a \emph{modular geodesic}.
Modular geodesics are in one-to-one correspondence with conjugacy classes of hyperbolic elements in $\PSL(2,\ZZ)$, i.e.\ those with trace more than two.
Note that $R$ and $L$ are parabolic elements, with trace two, but
any word in positive powers in $R$ and $L$ involving both letters is hyperbolic.

A modular geodesic lifts to $\HH^2$, tiled by the Farey tessellation. Series observed that such lifts cut out a sequence of triangles~\cite{Series:ModularSurface}. Within a given triangle an oriented geodesic enters through one side and then either exits through the side on its left (cutting off a single ideal vertex on its left side) or exits to its right. The sequence of rights and lefts determines a word in positive powers of $R$ and $L$ up to cyclic order called the \emph{cutting sequence}. This agrees with the matrix product corresponding to the geodesic.

Now consider the unit tangent bundle of the modular surface, $\UT(\SigmaMod)$. This is a Seifert fibred space whose base orbifold is $\SigmaMod$, with cone points of orders two and three and a cusp. 
In~\cite{Milnor}, Milnor proves that $\UT(\SigmaMod)$ is homeomorphic to the complement of the trefoil in $S^3$, which proof he credits to Quillen. A neighbourhood of the cusp point of $\SigmaMod$ lifts to give a neighbourhood of the trefoil. 
By work of Ghys~\cite{Ghys}, {for any finite collection of closed geodesics on the modular surface, their canonical lifts can be jointly isotoped} in $\UT(\SigmaMod)$ to lie on the branched surface shown in Figure~\ref{Fig:BranchedSurface}. 
These are called \emph{modular links}. 

\begin{figure}
\centering
\begin{overpic}[height=4cm]{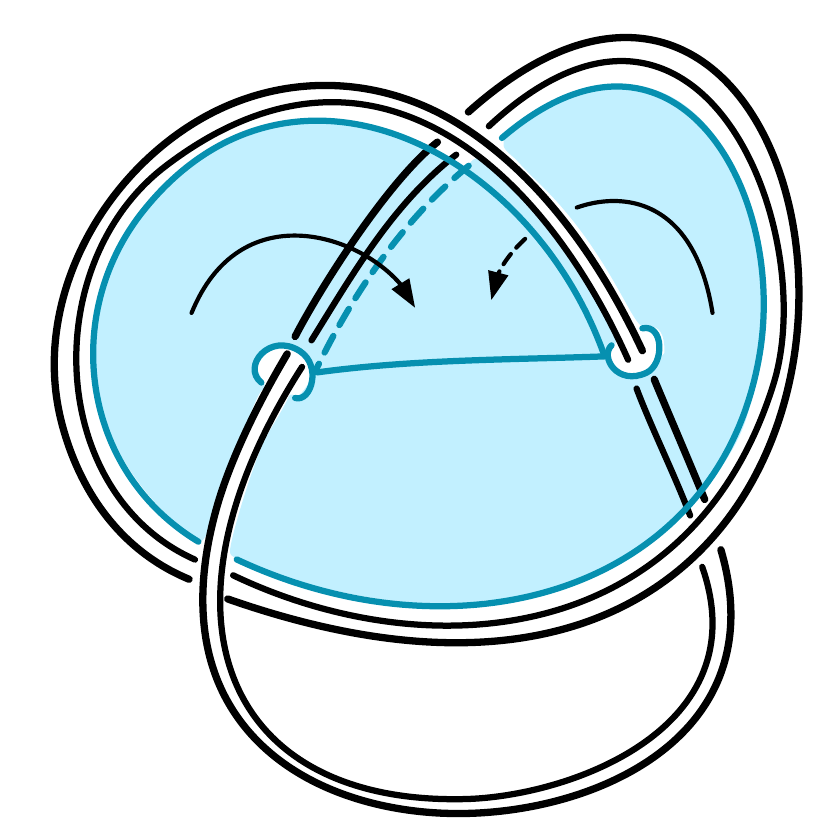}
    \put(2,75){{$R$}}
    \put(97,80){{$L$}}
\end{overpic}
\caption{The branched surface inside the complement of the trefoil, {with the direction of the semiflow indicated, pointing downwards from the branchline.}}
\label{Fig:BranchedSurface}
\end{figure}

A modular link follows two lobes of the branched surface, one on the right and one on the left, and it is determined up to cyclic permutation by the word in the letters $L$ and $R$. Thus the complement of a modular link corresponding to an $n$ component geodesic on the modular surface will be homeomorphic to the complement of a link in $S^3$ with $n+1$ components, with the additional component corresponding to the trefoil. Examples are shown in Figure~\ref{Fig:ModularArithmeticExamples}.

\subsection{The once-punctured torus}
Begin with the closed torus with no punctures, which we will denote by $\Sigma_{1,0}$: the surface of genus one with zero punctures. Once we fix a choice of generators $1/0$ and $0/1$ for $\pi_1(\Sigma_{1,0})$, any simple closed curve on the torus is determined by an element of $\QQ\cup\{1/0\}$. A geodesic representative of $p/q$ has constant tangent vector; the curve lifts to a line of constant slope $p/q$ in the universal cover $\RR^2$. 

The unit tangent bundle $\UT(\Sigma_{1,0})$ in this case is homeomorphic to $\Sigma_{1,0}\times S^1$. 
For ease of notation, we will write a point $e^{i t}$ in $S^1$ simply as $t$; in this form, two points in $S^1$ are equivalent if they differ by addition of an integer multiple of $2\pi$. 
Then the canonical lift of a curve $\gamma$ of slope $p/q$ is a curve $\gamma \times \{\arctan(p/q)\}$ when oriented with tangent vector pointing towards $e^{i \arctan(p/q)}$ in $\CC$. The curve has two orientations; when oriented in the opposite direction the canonical lift becomes $\gamma\times\{\arctan(p/q) + \pi\}$. Note that in either case, it has constant second coordinate. (Also note this discussion needs to be modified appropriately for $p/q=1/0$; we leave that to the reader.)

Now consider the once-punctured torus, which we denote by $\Sigma_{1,1}$: the genus one surface with one puncture. Consider the abelian cover of the punctured torus; for now we view this as the plane $\RR^2$ with integer lattice points removed. {The line $y=0$ in $\RR^2$ projects to an arc $\mu$ on $\Sigma_{1,1}$ with both endpoints on the puncture. Similarly, the line $x=0$ projects to an arc $\lambda$. }

Consider those simple closed curves on the punctured torus that are parallel to lines in $\RR^2$ of rational slope $p/q$, but disjoint from points on the integer lattice. These lines of rational slope project to closed curves in $\Sigma_{1,1}$ meeting {$\mu$ a total of $|p|$ times, and meeting $\lambda$ a total of $|q|$ times.} We let $p/q$ denote the closed curve. In particular, {a closed curve parallel to $\mu$ is $0/1$, and one parallel to $\lambda$ is $1/0$.} Note these are not all the closed curves in $\Sigma_{1,1}$; we are omitting curves that wrap around the puncture in more complicated ways. However, these are the closed curves we will encounter in this paper. 

Now consider the canonical lifts of such curves. The unit tangent bundle of the punctured torus is homeomorphic to the product $\Sigma_{1,1}\times S^1$. Just as for the closed torus, up to homeomorphism, the canonical lift of a curve of slope $p/q$ in $\UT(\Sigma_{1,1})$ has the form $\gamma \times \{\arctan(p/q)\}$ oriented in one direction, or $\gamma\times\{\arctan(p/q) + \pi\}$ oriented in the other direction. That is, in either case we may isotope $p/q$ in $\Sigma_{1,1}$ to have constant tangent vector. 

In addition to the unit tangent bundle one may consider the projective tangent bundle $\PT(\Sigma_{1,1})$, where one quotients out by the action of $\pm1$ on $S^1$, i.e.\ antipodal points are identified. The unit tangent bundle is a degree two cover of the projective tangent bundle. The two lifts of any fixed geodesic are identified in the quotient, hence an unoriented closed geodesic has a unique lift to the projective tangent bundle. Furthermore, its complement in the projective tangent bundle is covered via a degree two covering map by the complement of both its lifts in the unit tangent bundle.

\begin{lemma}\label{Lem:6FoldCover}
The punctured torus forms a 6-fold cover of the modular surface. The group of covering transformations is generated by a rotation of order three and a rotation of order two. 

Similarly, the unit tangent bundle of the punctured torus forms a 6-fold cover of the unit tangent bundle of the modular surface. The group of covering transformations is generated by two glide rotations, of orders three and two. 
\end{lemma}

\begin{proof}
We will study the cover $\Sigma_{1,1} \to \SigmaMod$ by considering first the abelian cover $\RR^2\setminus \Lambda \to \Sigma_{1,1}$, where $\Lambda$ is a lattice, and showing that $\SigmaMod$ is obtained as a further quotient of this space. 

Triangulate $\Sigma_{1,1}$ by adding the edges {$\lambda$ parallel to $1/0$ and $\mu$ parallel to $0/1$ as above}, and an arc parallel to the slope $1/1$. This subdivides $\Sigma_{1,1}$ into two triangles, which we view as equilateral triangles. The abelian cover of $\Sigma_{1,1}$ can then be viewed as obtained by tiling $\RR^2$ by these equilateral triangles, and removing all vertices to form the lattice $\Lambda$. 
We obtain $\Sigma_{1,1}$ by taking the quotient of $\RR^2\setminus\Lambda$ by covering transformations that translate in the direction of $\mu$ and $\lambda$. 

To obtain $\SigmaMod$, we quotient further, first by a rotation by $2\pi/3$, fixing the centre of one of the 
equilateral triangles and rotating
its three vertices {(the second triangle will also be rotated around its centre as a result)}, 
and then by a rotation by $\pi$, fixing the centre of an edge of an equilateral triangle and rotating that edge back to itself, swapping its endpoints {and swapping the two triangles (note this will rotate simultaneously the other two edges about their centres). These two rotations generate a group of order 6,} and the quotient is $\SigmaMod$. See Figure~\ref{Fig:6FoldCover}.

\begin{figure}%[h]
\centering
\begin{overpic}[height=4cm]{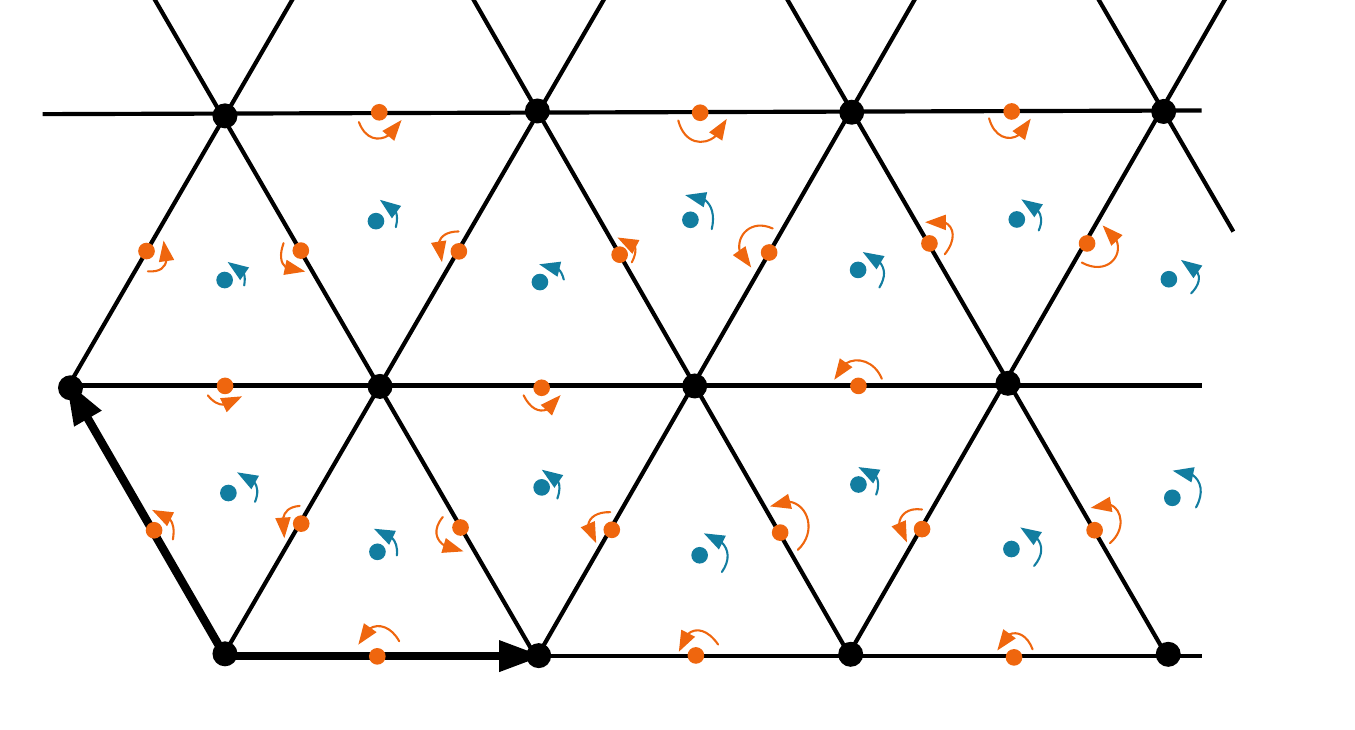}
    \put(2,15){{$\lambda$}}
    \put(25,0){{$\mu$}}
\end{overpic}
\caption{Taking the quotient of $\RR^2\setminus\Lambda$ by translations gives $\Sigma_{1,1}$. Quotient further by $2\pi/3$ rotations about centres of 
triangles, and $\pi$ rotations about centres of edges to obtain $\SigmaMod$. }
\label{Fig:6FoldCover}
\end{figure} 

Now consider the unit tangent bundles. The unit tangent bundle $\UT(\Sigma_{1,1})$ is a trivial product, so it is covered by $(\RR^2\setminus\Lambda)\times\RR$. We obtain $\UT(\Sigma_{1,1})$ by taking the quotient by translations on $\RR^2$ in the directions of $\mu$ and $\lambda$, and by a translation $(x,y, 0)\mapsto (x,y, 2\pi)$ in the $\RR$ direction. 

To obtain $\UT(\SigmaMod)$, further quotient by a covering transformation of order three, and one of order two. The first is the glide rotation $V$ that rotates an equilateral triangle in $\RR^2$ by $2\pi/3$ about its centre, and translates it in the $\RR$ direction by $2\pi/3$. Then $V$ has order three in $\Sigma_{1,1}\times S^1 = \UT(\Sigma_{1,1})$. 
The second is the glide rotation $U$ that rotates $\RR^2$ by $\pi$ in the centre of an edge of an equilateral triangle, and shifts in the $\RR$ direction by $\pi$. This has order two in $\UT(\Sigma_{1,1})$. Observe it takes the canonical lift of an oriented curve in $\Sigma_{1,1}$ to the canonical lift of the oppositely oriented curve.

We claim that the quotient of $\UT(\Sigma_{1,1})$ by $U$ and $V$ is $\UT(\SigmaMod)$. To see this, note that the quotient is Seifert fibred, with base orbifold a sphere with one cusp, one cone point of order two, and one cone point of order three. This is homeomorphic to $\UT(\SigmaMod)$. 
\end{proof}

\begin{remark}\label{rem: punctured surfaces}
More generally, any orientable 
hyperbolic surface with  at least one puncture can be tiled by ideal triangles. There is then a hyperbolic structure that allows us to identify its fundamental domain with a finite portion of the Farey tessellation of $\HH^2$. Since the modular group $\PSL_2(\RR)$ is the full symmetry group of the tessellation, this yields a representation of the surface's fundamental group as a subgroup of the modular group of finite index, and the surface is therefore a branched surface of $\SigmaMod$. Thus one can consider lifts of modular geodesics to any such surface, and as the unit tangent bundle is always trivial in this case,
if the lift is simple the situation will be similar.
\end{remark}

\subsection{Curves on the once-punctured torus and the Farey tessellation}

Isotopy classes of simple closed curves on the punctured torus are organised by the same Farey tessellation. Recall that the Farey complex can be considered as $\HH^2$ with boundary $\RR\cup \{1/0\}$. Isotopy classes of simple closed curves on $\Sigma_{1,1}$ correspond to points in $\QQ\cup \{1/0\}$. The geometric intersection number of curves $a/b$ and $c/d$ is given by $|ad-bc|$. When $a/b$ and $c/d$ intersect exactly once, 
they correspond to an edge in the Farey complex: a hyperbolic geodesic running from $a/b$ in $\QQ\cup \{1/0\}$ to $c/d$ in $\QQ\cup\{1/0\}$. 
We say such curves are \emph{Farey neighbours}. 
The matrix
$
\begin{pmatrix}
    a & c \\
    b & d 
\end{pmatrix}
$ 
in $\PSL(2,\ZZ)$ takes the edge between $1/0$ and $0/1$ to the edge between $a/b$ and $c/d$ in $\HH^2$. 

\begin{definition}\label{Def:FareyNeighbours}
We say an ordered collection of simple closed curves $\alpha_1, \dots, \alpha_n$ in $\Sigma_{1,1}$ are \emph{Farey neighbours} if each $\alpha_j$ and $\alpha_{j+1}$ are connected by an edge of the Farey triangulation, for $j=1, \dots, n-1$, and if $\alpha_n$ and $\alpha_1$ are also connected by an edge of the Farey triangulation. 
\end{definition}

\section{Arithmetic {K}leinian groups}

Let $K$ be a link in a compact $3$-manifold with torus boundary. Suppose that the interior of the complement has a complete hyperbolic structure, meaning it is isometric to $\HH^3/G$, where $\HH^3$ is the hyperbolic $3$-space and $G$ is a torsion-free, non-cocompact Kleinian group of finite covolume. The following definition of arithmeticity is a consequence of~\cite[Theorem 9.2.2]{MaclachlanReid}.

\begin{definition}\label{Def:Arithmetic} 
A non-cocompact Kleinian group is \emph{arithmetic} if it is conjugate in $\PSL(2,\CC)$ to a group commensurable with $\PSL_2(O_d)$, where $O_d$ is the ring of integers in the imaginary quadratic number field $\QQ(\sqrt{-d})$, with $d$ a positive integer. 
Such a group $\PSL(2,O_d)$ is called a \emph{Bianchi group}. We say that a hyperbolic 3-manifold is \emph{arithmetic} if the corresponding Kleinian group is arithmetic. Similarly, a knot or link with arithmetic complement is said to be arithmetic. 
\end{definition} 

An example of an Bianchi group is the group $\PSL(2,\mathbb{Z}[i])$, called the Picard group.
The Picard group is generated by face pairings of a fundamental region
\[ \mathcal{F} =\{(x,y,t)\in \HH^3 \mid x^2+y^2+t^2\geq 1, -1/2 \leq x \leq 1/2, 0 \leq y \leq 1/2\};
\]
see~\cite[\S~1.4.1]{MaclachlanReid}. This is a quotient of a regular ideal octahedron. 
In fact, analogous to the two dimensional case, $\HH^3$ is tessellated by regular ideal octahedra, with ideal vertices at all points of $\QQ[i]$. The Picard group $\PSL(2,\ZZ[i])$ is a subgroup of index two of the full symmetry group of this tessellation. Thus we have the following well-known result; see \cite[\S~9.4]{MaclachlanReid}, \cite{NeumannReid:Orbifolds} and also \cite{abe2013geometry}. 

\begin{lemma}\label{Lem:OctahedraArithmetic}
Any finite volume hyperbolic 3-manifold obtained by gluing regular ideal octahedra is arithmetic. 
\end{lemma}

Note that arithmeticity is preserved by taking finite covers or quotients; any space that is finitely covered by such a space is also arithmetic.

\section{Regular octahedra for neighboring slopes}

We now return to curves on the punctured torus $\Sigma_{1,1}$. In this section, we build arithmetic links in $\UT(\Sigma_{1,1})$. 

\begin{lemma}\label{Lem:Octahedron}
Suppose $\alpha$ and $\beta$ are two simple closed curves on the punctured torus $\Sigma_{1,1}$ that share an edge in the Farey triangulation. Let $N_{\alpha,\beta}$ denote the space obtained from  $\Sigma_{1,1}\times [0,1]$ by removing $\alpha$ from $\Sigma_{1,1}\times\{0\}$ and removing $\beta$ from $\Sigma_{1,1}\times\{1\}$. Then $N_{\alpha,\beta}$ admits a complete hyperbolic structure obtained by gluing in pairs the {eight} faces of a regular ideal octahedron. 
\end{lemma}

\begin{proof}
When $\alpha=1/0$ and $\beta=0/1$, this is well known and is illustrated in Figure~\ref{Fig:Octahedron}; see, for example, \cite[Lemma~2.4]{HuiPurcell:RodPackings}. 
On the left of that figure, $\Sigma_{1,1}\times [0,1]$ is obtained by gluing the front face to the back, and the left face to the right. 

On the right of the figure, observe that this gluing now identifies the front and back triangles opposite each other across the ideal vertex at the top of the octahedron, and the left and right triangles opposite each other across the ideal vertex at the bottom of the octahedron. If we give the ideal octahedron the hyperbolic geometry of a regular ideal octahedron, then each edge is identified to two edges of the ideal octahedron. The {remaining unglued }top and bottom faces become totally geodesic once-punctured annuli.

\begin{figure}
    \centering
    \includegraphics{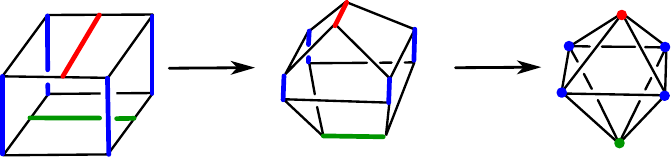}
    \caption{Starting on the left with $\Sigma_{1,1}\times [0,1]$ with $\alpha=1/0$ drilled from $\Sigma_{1,1}\times\{0\}$ and $\beta=0/1$ drilled from $\Sigma_{1,1}\times\{1\}$, obtain a regular ideal octahedron on the right. }
    \label{Fig:Octahedron}
\end{figure}

For general $\alpha=p/q$ and $\beta=r/s$, $\alpha$ and $\beta$ are Farey neighbours when $|ps-qr|=1$. In this case there exists a homeomorphism from $N_{0,\infty=1/0}$ to $N_{\alpha, \beta}$ induced by the action of the linear automorphism 
$\begin{pmatrix}
    p & r \\
    q & s
\end{pmatrix}$
taking $\Sigma_{1,1}\times\{t\}$ to $\Sigma_{1,1}\times\{t\}$ for all $t$, and taking $(\Sigma_{1,1}\times\{0\})\setminus\{1/0\}$ to $(\Sigma_{1,1}\times\{0\})\setminus\alpha$ and $(\Sigma_{1,1}\times\{1\})\setminus\{0/1\}$ to $(\Sigma_{1,1}\times\{1\})\setminus\beta$. This can be realised by a hyperbolic isometry. 
\end{proof}

\begin{theorem}\label{Thm:FareyPath}
Let $\alpha_1, \dots, \alpha_n$ be simple closed curves in $\Sigma_{1,1}$ that are Farey neighbours. Drill $\Sigma_{1,1}\times S^1$ by removing $\alpha_j$ from $\Sigma_{1,1}\times \{ j/n\}$. The resulting manifold has a complete hyperbolic structure obtained by gluing $n$ regular ideal octahedra. 
\end{theorem}

\begin{proof}
Cut the drilled manifold along each surface $\Sigma_{1,1}\times \{j/n\}$. Obtain blocks of the form $N_{\alpha_j, \alpha_{j+1}}$. By Lemma~\ref{Lem:Octahedron}, each of these can be given the hyperbolic structure of a regular ideal octahedron, with two top faces unglued and two bottom faces unglued. 

Glue the top faces of $N_{\alpha_j, \alpha_{j+1}}$ to the bottom faces of $N_{\alpha_{j+1}, \alpha_{j+2}}$ for $j=1, \dots, n$ modulo $n$. The gluing will be by the identity, along totally geodesic once-punctured annuli. These have a unique hyperbolic structure, hence the gluing is by isometry. 

We claim this gives a complete hyperbolic structure on the original drilled manifold. The proof is by the Poincar{\'e} polyhedron theorem; see Epstein and Petronio~\cite{EpsteinPetronio} for a careful exposition. 
The gluing identifies blocks $N_{\alpha_j, \alpha_{j+1}}$ top to bottom, yielding a manifold homemomorphic to the desired manifold. Under the gluing, each edge is 4-valent. 
Thus when edges are glued, the monodromy around any edge is the identity: formed by gluing four right dihedral angles. This is sufficient to ensure that the manifold has a (possibly incomplete) hyperbolic structure. For completeness, notice that in the boundary of a horoball neighbourhood of any cusp, we identify a sequence of truncated neighbourhoods of the ideal vertices; these are squares. The squares are glued to obtain a tiling of the horospherical torus. Thus the regular ideal octahedra induce a Euclidean structure on each cusp. It follows that the hyperbolic metric obtained from the octahedra is a complete metric on the drilled manifold; see also~\cite[Theorem~4.10]{Purcell:HypKnotTheory}. 
\end{proof}

We wish to apply Theorem~\ref{Thm:FareyPath} to a result about canonical lifts of Farey neighbours in the unit tangent bundle $\UT(\Sigma_{1,1})$. However, we need to take some care in orienting the curves. As noted above, each curve $\gamma=p/q$ has two orientations. For one orientation, the canonical lift $\widehat{\gamma}$ will lie in $\Sigma_{1,1}\times\{\arctan(p/q)\}$ and the other will lie in $\Sigma_{1,1}\times\{\arctan(p/q)+\pi\}$.
The canonical lift $\widetilde{\gamma}$ to the projective tangent bundle $\PT{\Sigma_{1,1}}$ (which is the same trivial bundle $\Sigma_{1,1}\times S^1$) is well defined.

\begin{theorem}\label{Thm:CanonicalLiftFareyNeighbours}
Let $\Gamma:=\{{\gamma_j = a_j/b_j}\}_{j=1}^n$ be a collection of simple closed geodesics on the punctured torus made of Farey neighbours, with each $\gamma_j$ oriented in the direction of $\exp(i\arctan(a_j/b_j))$. Let $\overline{\Gamma}:=\{\overline{\gamma_j}\}_{j=1}^n$ be the same collection, with each curve oriented in the opposite direction. Then:
\begin{enumerate}
\item $\UT(\Sigma_{1,1})\setminus{\widehat{{\Gamma}}}\cong\UT(\Sigma_{1,1})\setminus{\widehat{\overline{\Gamma}}}\cong\PT(\Sigma_{1,1})\setminus{\widetilde{\Gamma}}$ is arithmetic, obtained by gluing $n$ regular ideal octahedra.
\item $\UT(\Sigma_{1,1})\setminus(\widehat{\Gamma} \cup \widehat{\overline{\Gamma}})$ is arithmetic, obtained by gluing $2n$ regular ideal octahedra.
\end{enumerate}
\end{theorem}

\begin{proof}
Each $\gamma_j = a_j/b_j$ corresponds to a distinct slope in $\QQ\cup \{1/0\}$. We may assume the $b_j$ are nonnegative integers. By our orientation convention, each curve $\widehat{\gamma_j}$ will be drilled from $\Sigma_{1,1}\times\{\arctan(a_j/b_j)\} \subset \Sigma_{1,1}\times S^1$. Because $\Gamma$ is a collection of Farey neighbours, there is some minimal slope in $\QQ$, which we may relabel to be $\gamma_1=a_1/b_1$, and then up to relabeling, the slopes satisfy $a_1/b_1 < a_2/b_2 < \dots < a_n/b_n$. Then when we drill, the curves are drilled in cyclic order $\gamma_1$, $\gamma_2$, up to $\gamma_n$ in the $S^1$ factor of $\Sigma_{1,1}\times S^1$. The drilling is therefore homeomorphic to the drilling of Theorem~\ref{Thm:FareyPath}. Then the fact that $M_{\widehat{{\Gamma}}}$ is obtained by gluing $n$ regular ideal octahedra follows from Theorem~\ref{Thm:FareyPath}, and the fact that it is arithmetic follows from Lemma~\ref{Lem:OctahedraArithmetic}. An identical argument holds for $\overline{\Gamma}$. 

As for the union of $\widehat{\Gamma}$ and $\widehat{\overline{\Gamma}}$, the arithmeticity follows from the fact it is a double cover of  $\PT(\Sigma_{1,1})\setminus{\widehat{{\Gamma}}}$. Furthermore, the first $n$ canonical lifts will be at heights $\arctan(a_1/b_1) < \dots < \arctan(a_n/b_n)$, and the next $n$ at $\arctan(a_1/b_1)+\pi$ through $\arctan(a_n/b_n)+\pi$. Thus again we drill the Farey neighbours in an order homeomorphic to that of Theorem~\ref{Thm:FareyPath}, and so that theorem implies that the complement is built of $2n$ regular ideal octahedra. 
\end{proof}

\section{Projecting and lifting on the modular surface}

\begin{lemma}\label{Lem:ProjectAndLift}
Let $\gamma$ be an oriented geodesic on the modular surface $\SigmaMod$, obtained by projecting the simple closed curve $p/q \subset \Sigma_{1,1}$ via the covering map of Lemma~\ref{Lem:6FoldCover}. Then under the covering map, $\gamma$ has six lifts in $\Sigma_{1,1}$. 
These are $p/q$, $q/(q-p)$, $(p-q)/p$, and each of these three curves oriented in the opposite direction: $\overline{p/q}$, $\overline{q/(q-p)}$, and $\overline{(p-q)/p}$. 
\end{lemma}

\begin{proof}
We consider the images of $p/q$ under the rotations of order two and three of Lemma~\ref{Lem:6FoldCover}. As in the proof of that lemma, we will view $\Sigma_{1,1}$ as a quotient of the tiling of $\RR^2$ by equilateral triangles with vertices removed. 

Recall that the rotation of order three rotates an ideal triangle, permuting its vertices. 
Consider its effect on the curve $p/q$. We may assume without loss of generality that $q\geq 0$. If $p\geq 0$, then the curve $p/q$ meets the side $\mu$ of an equilateral triangle in the fundamental domain for $\Sigma_{1,1}$ a total of {$p$ times}. It meets $\lambda$ a total of {$q$ times}, and meets the diagonal $|q-p|$ times. See Figure~\ref{Fig:RotateV}, which shows the case {$q>p>0$}. 

\begin{figure}[ht!]
\centering
\begin{overpic}[width=11cm]{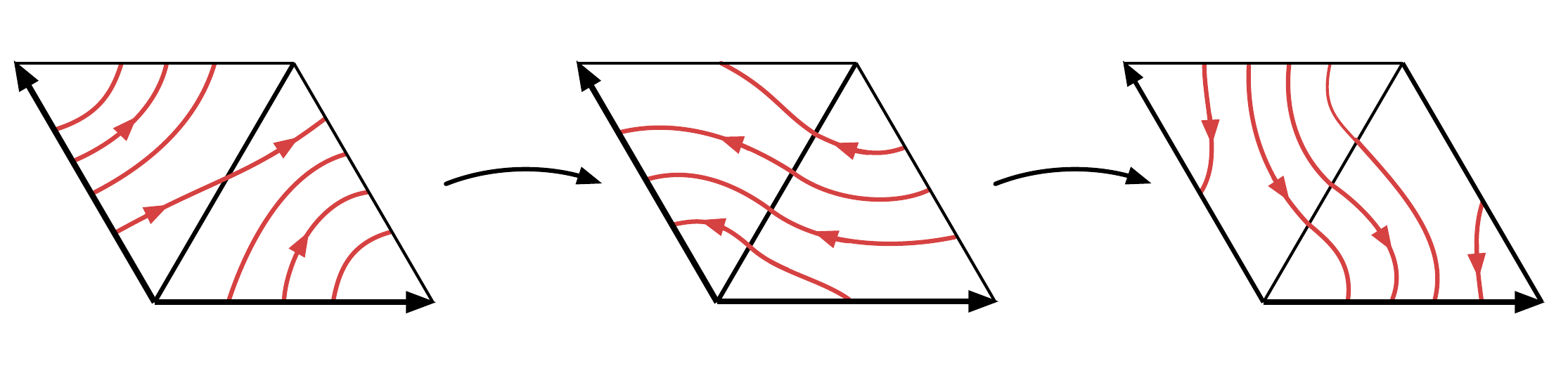}
    \put(4,8){{$\lambda$}}
    \put(17,1){{$\mu$}}
\end{overpic}
\caption{A rotation by $2\pi/3$ about the centre of each equilateral triangle takes the curve $p/q$ to the curve {$\overline{(p-q)/p}$}, and a further rotation takes it to {$\overline{q/(q-p)}$}. Shown is the case {$q>p>0$}. Similar pictures give other cases.}
\label{Fig:RotateV}
\end{figure}

Rotating by $2\pi/3$ takes the curve to one meeting $\mu$ a total of {$|q-p|$} times, meeting $\lambda$ a total of {$|p|$} times, and meeting the diagonal {$q$} times. In case $q>p> 0$, as shown in Figure \ref{Fig:RotateV}, the resulting slope is negative, of value {$\overline{(p-q)/p}$,} and a further rotation gives the curve of slope {$\overline{q/(q-p)}$.}
{Our convention is to take an overline if the curve crosses lambda from right to left; we will see that in all cases we obtain each curve in both direction so this convention will not matter.}

{If $p>q> 0$, the result of the rotation is positive, of slope $\overline{(p-q)/p}$},
and a further rotation results in a curve of slope {${q/(q-p)}$.} 

If $p<0$ then the curve $p/q$ meets $\mu$ a total of $|p|$  times, meets $\lambda$ a total of $q$ times, and meets the diagonal $|p|+q = q-p$ times. 
The resulting slopes after rotating are {${(q-p)/p}$} and {${\overline{(p-q)/p}}$}. 

Finally if one of $p$ or $q$ is zero, or $p=q=1$, the three slopes up to rotation are $0/1$, $1/0$, and $1/1$, and the lemma holds for these. 

Now consider the rotation of order two, with fixed point on an edge of the triangle. This takes the $p/q$ curve back to itself, but it gives it the opposite orientation. This will give us the curve $\overline{p/q}$. Similarly it gives the other two curves with opposite orientations. {Thus in all cases we obtain the set of both orientations of each of the slopes $\{p/q, {q/(q-p)},{(p-q)/p}\}$, as required.} 
\end{proof}

The following lemma shows that in lieu of rotating the closed geodesics and then considering the resulting slopes as above, one may instead directly rotate the slopes along the circle at infinity.

\begin{lemma}\label{Lem:RotationV}
For any $p/q\in \QQ\cup\{1/0\}$, $V(p/q) = q/(q-p)$ and $V^2(p/q)=(p-q)/p$. 
\end{lemma}

\begin{proof}
Recall that $V$ is the matrix 
$V=\pm\begin{pmatrix} 0 & -1\\ 1 & -1 \end{pmatrix}$. 
Then $V(p/q) =q/(q-p)$, and $V^2(p/q)=(p-q)/p$. 
\end{proof}
{Observe that the first rotation shown} in Figure~\ref{Fig:RotateV} is $V^2$.

We will now turn a sequence of geodesics in the punctured torus into a sequence of geodesics on the modular surface. We start with an example, shown in Figure~\ref{Fig:Farey}. Consider the $3/2$ curve. There is a shortest path from the Farey triangle with vertices $(1/1, 0/1, 1/0)$ to a Farey triangle with vertex $3/2$. The path meets three Farey triangles, with vertices $(1/1, 0/1, 1/0)$, $(1/1, 2/1, 1/0)$, and $(1/1, 2/1, 3/2)$. Form a collection of curves $\Gamma$ by adding all the distinct slopes in all these triangles to $\Gamma$. 

Thus $\Gamma$ consists of $0/1$, $1/0$, {$1/1$}, $2/1$, and $3/2$. Note these are Farey neighbours, so Theorem~\ref{Thm:CanonicalLiftFareyNeighbours} implies that the complement of their canonical lifts (oriented both ways) is an arithmetic manifold. 

\begin{figure}[ht!]
    \centering
    \begin{overpic}[width=5cm]{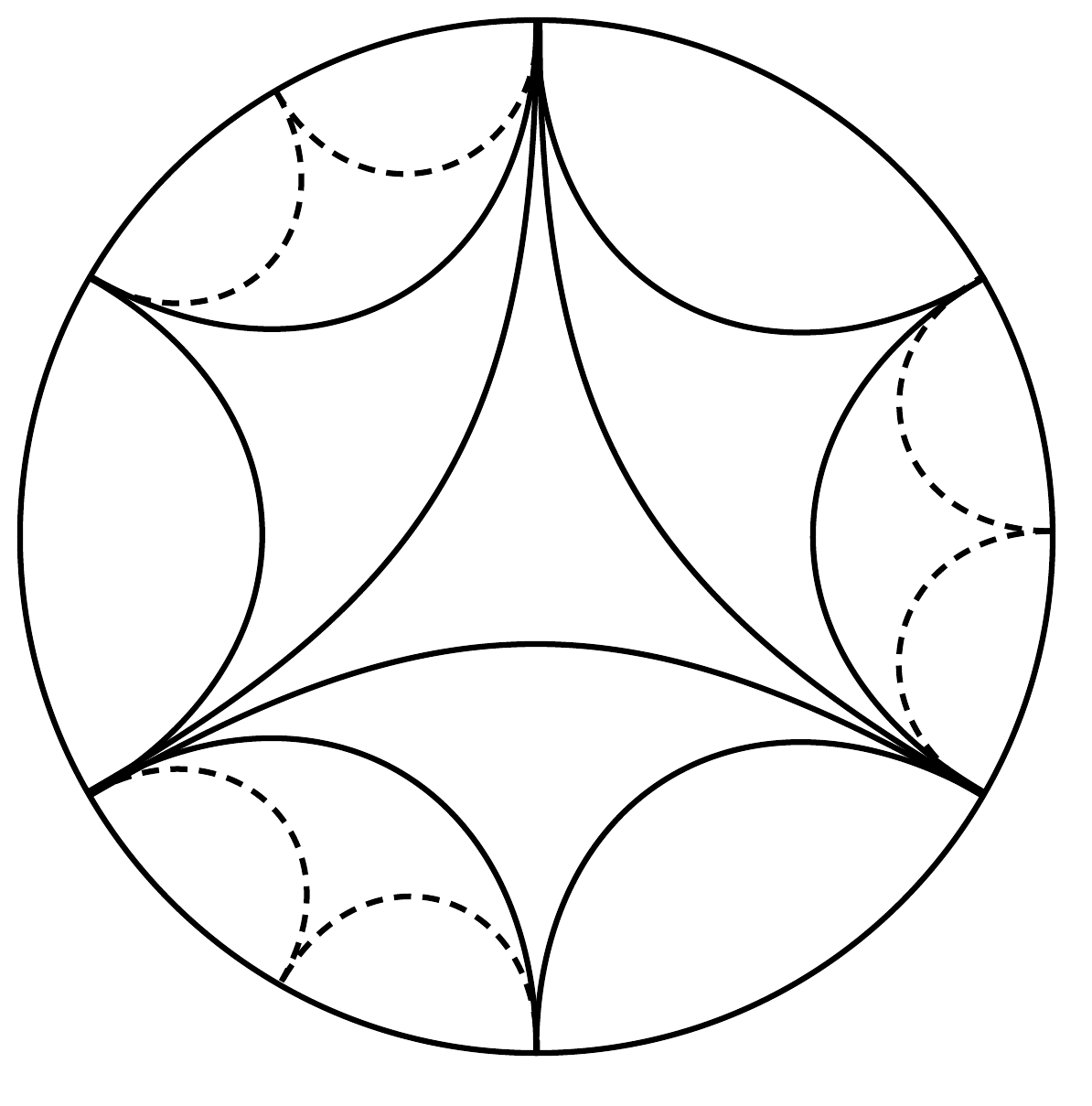}
    \put(92,23){$1$}
    \put(2,23){$0$}
    \put(46,101){$\infty$}
    \put(47,-3){$\frac{1}{2}$}
    \put(94,74){$2$}
    \put(-4,73){$-1$}
    \put(98,50){$3/2$}
    \end{overpic}
    \caption{The Farey graph of rational slopes.}
    \label{Fig:Farey}
\end{figure}
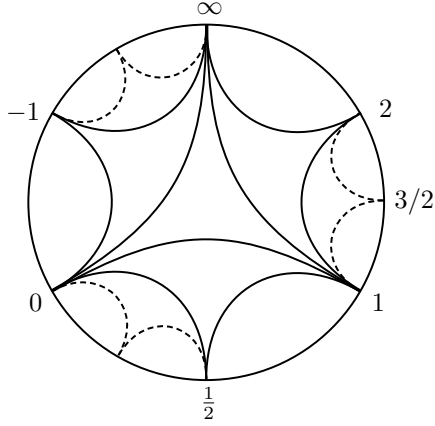

We wish to apply the covering projection from $\UT(\Sigma_{1,1})$ to $\UT(\SigmaMod)$. However, note that the canonical lift of $\Gamma$ does not cover any link complement in the unit tangent bundle of the modular surface, because $\Gamma$ does not contain all the 
preimages of its projections to the modular surface. Thus we extend $\Gamma$, by including all images of $\Gamma$ under the rotations $V$ and $V^2$. Thus in the example of Figure~\ref{Fig:Farey}, we would add $-1/1 = V(2/1)$, $1/2=V^2(2/1)$, $-2/1=V(3/2)$ and $1/3=V^2(3/2)$. The result is a again a collection of Farey neighbours, and now the complement of all canonical lifts is a cover of the complement of a modular link. 

We generalise this example. 

\begin{theorem}\label{Thm:ModularArithmetic1}
Any modular geodesic that lifts to a simple closed curve $\alpha$ on the once punctured torus is part of an arithmetic link in $\UT(\SigmaMod)$ with all components being modular geodesics. 
Moreover, suppose the shortest path in the Farey triangulation between the triangle $(0,1,\infty)$ and any triangle with vertex $\alpha$ passes through $x$ Farey triangles. Then the complement of the lift can be decomposed into $x$ regular ideal octahedra.
\end{theorem}

\begin{proof}
Given any slope $p/q$, there is a shortest path in the Farey triangulation from the centre of the triangle with vertices $(0/1, 1/1, 1/0)$ to a triangle with a vertex $p/q$. This will pass through some number of Farey triangles. Build a collection of curves $\Gamma$ by adding all the slopes corresponding to all the vertices of the Farey triangles in the path. Thus $\Gamma$ will contain $0/1, 1/1, 1/0$ and $p/q$, as well as additional curves at vertices of Farey triangles. 
At this step, $\Gamma$ will contain a total of $2+x$ slopes: three corresponding to the first triangle $(0/1, 1/1, 1/0)$, and $x-1$ additional slopes, one for each new triangle in the path. 

Next, expand $\Gamma$ by adding all images of $\Gamma$ under the rotations $V$ and $V^2$ of Lemma~\ref{Lem:RotationV}. Note this adds $2(x-1)$ additional slopes to $\Gamma$, so that in total, $\Gamma$ now contains $3x$ slopes. 

Observe that the collection $\Gamma$ can now be ordered in $\QQ \cup \{1/0\}$ to give a set of Farey neighbours, invariant under the action of $V$. 
Theorem~\ref{Thm:CanonicalLiftFareyNeighbours} then implies that $\UT(\Sigma_{1,1})\setminus(\widehat{\Gamma}\cup\widehat{\overline{\Gamma}})$ is arithmetic, obtained by gluing $6x$ regular ideal octahedra. By Lemma~\ref{Lem:ProjectAndLift}, the drilled curves are exactly the canonical lifts of all curves projecting to a collection of $x$ simple closed curves on the modular surface. 

Now consider the action of the covering tranformations of Lemma~\ref{Lem:6FoldCover} from $\UT(\Sigma_{1,1})$ to $\UT(\SigmaMod)$. By construction, the order two transformation will take the canonical lift of $p/q$ to that of $\overline{p/q}$. The order three transformation will take the canonical lift of $p/q$ to $V(p/q)$ and $V^2(p/q)$. Thus  $\UT(\Sigma_{1,1})\setminus(\widehat{\Gamma}\cup\widehat{\overline{\Gamma}})$ is a six-fold cover of the complement of a collection of canonical lifts in $\UT(\SigmaMod)$. 

Finally, observe that each of the covering transformations maps a regular ideal octahedron to a distinct regular ideal octahedron. By the construction of Theorem~\ref{Thm:FareyPath}, the regular ideal octahedra lie between canonical lifts that share an edge in the Farey triangulation. The covering transformation of degree three takes the octahedron between $a/b$ and $c/d$ to that between $V(a/b)$ and $V(c/d)$, and then again to that between $V^2(a/b)$ and $V^2(c/d)$; these are all distinct edges of the Farey triangulation. The covering transformation of degree two takes the octahedron between $a/b$ and $c/d$ to that between $\overline{a/b}$ and $\overline{c/d}$; this octahedron differs from the original by a rotation by $\pi$ in the $S^1$ factor of $\UT(\Sigma_{1,1}) \cong \Sigma_{1,1}\times S^1$. 

Then when we take the quotient by covering transformations, we obtain an arithmetic canonical link complement in $\UT(\SigmaMod)$, with the link containing the original curve, and built from $6x/6 = x$ regular ideal octahedra. 
\end{proof}

Theorem~\ref{Thm:MainArithmetic} from the introduction is an immediate consequence. 

\begin{corollary}\label{Cor:InfinitelyMany}
There are infinitely many arithmetic modular links. \qed
\end{corollary}

\section{Cutting sequences}

As explained in Section~\ref{Sec:Surfaces}, canonical lifts of geodesics in $\UT(\SigmaMod)$ can be viewed as links in $S^3\setminus K$ where $K$ is the trefoil knot. In the previous section, we found infinitely many arithmetic canonical link complements. We wish to identify these links as the complement of links in the 3-sphere. To do so, we will find cutting sequences for the links, enabling us to identify them in the branched surface of Figure~\ref{Fig:BranchedSurface} following \cite{Ghys}. That is the main goal of this section.

\begin{definition}
Let $\alpha$ be a closed geodesic in the modular surface $\SigmaMod$. The \emph{$LR$-cutting sequence} of $\alpha$ is the bi-infinite sequence of instances of $L$ and $R$ obtained as follows. Recall that a fundamental domain for $\SigmaMod$ is the quotient of an ideal triangle by an order three and an order two rotation. As in the proof of Lemma~\ref{Lem:6FoldCover}, we take a cover of $\SigmaMod$ that tiles $\RR^2$ by equilateral triangles, and remove the lattice $\Lambda$ consisting of the vertices of these triangles. Lift $\alpha$ to this cover. Consider a point of intersection of $\alpha$ with an edge of a triangle. Then in the adjacent triangle, $\alpha$ either runs next to the edge to the left or to the right. If it runs to the left, take the letter $L$. If it runs right, take the letter $R$. Now repeat for the next triangle, and so on. Because $\alpha$ is a closed geodesic, eventually $\alpha$ returns to an edge identified with the original edge of intersection, and the sequence will repeat. 
\end{definition}

\begin{remark}\label{rem:positive slopes}
Since different lifts of the geodesic $\alpha$ differ by an element of $\PSL(2,\ZZ)$ which preserves the Farey tessellation by ideal triangles, the cutting sequence remains the same up to cyclic order no matter which lift of $\alpha$ we start with. By reversing the orientation of $\alpha$ if necessary we may always assume its cutting sequence begins with an $L$. Furthermore, we may always assume it enters the $0,1,\infty$ triangle through the imaginary axis (oriented to the right) by using the rotation about $i$ given by $U$ above.
\end{remark}

We can similarly define a cutting sequence for simple closed curves in $\Sigma_{1,1}$. Take a curve $p/q$ with $p/q$ positive, and lift to the abelian cover of $\Sigma_{1,1}$ that we build by tiling $\RR^2$ with equilateral triangles, again as in the proof of Lemma~\ref{Lem:6FoldCover}. Lift $p/q$ to this cover. The lift will intersect lifts of the arcs $\mu$ and $\lambda$. If it intersects $\mu$, assign an instance of $A$. If it intersects $\lambda$, assign an instance of $B$. This gives an $AB$-cutting sequence for geodesics on $\Sigma_{1,1}$. 

The following algorithm, from Series~\cite{Series:GeometryMarkoff} and Davis~\cite[Algorithm~7.6]{Davis:Dynamics}, gives the $AB$-cutting sequence in terms of the continued fraction expansion of $p/q$.  

\begin{algorithm}\label{Alg:CuttingSequence}
\begin{enumerate}
\item Start with an infinite string consisting of incidences of the letter $A$. This corresponds to a lift of a geodesic of slope $0$.
{If $p/q\neq 0$, take a continued fraction expansion of the slope $p/q$ of the form $[a_1, a_2, \dots, a_k]$ where all the $a_j$ are positive.}
\item Insert $a_k$ instances of the letter $B$ between each pair of letters $A$. The corresponding trajectory now has slope $a_k$. {If $p/q=a_k$ we are done.}
\item Else swap every $A$ to $B$ and vice-versa. The corresponding trajectory now has slope $1/a_k$.  {If $p/q=1/a_k$ we are done.}
\item Else insert $a_{k-1}$ instances of the letter $B$ between each pair of letters $A$. The corresponding trajectory now has slope $a_{k-1}+\frac{1}{a_k}$. {If we have reached $p/q$ we are done.}
\item Else reverse $B$ and $A$. The corresponding trajectory now has slope $\frac{1}{a_{k-1}+\frac{1}{a_k}}$. {If we have reached $p/q$ we are done.}
\item Else continue this process, ending by inserting $a_1$ instances of the letter $B$ between each pair of letters $A$. This yields the $AB$-cutting sequence corresponding to the fractional slope $[a_1, a_2, \dots, a_k]$.
\end{enumerate}
\end{algorithm}

We wish to find the $LR$-cutting sequence corresponding to a modular geodesic, and the $AB$-cutting sequence of Algorithm~\ref{Alg:CuttingSequence} for its lift to the once punctured torus. By Remark \ref{rem:positive slopes} the lift we choose does not change the $LR$-cutting sequence and thus we may choose the lift to be a curve of slope $p/q$ on $\Sigma_{1,1}$ where $p$ and $q$ are non-negative. We can obtain the $LR$-cutting sequence corresponding to its projection as follows.

\begin{algorithm}\label{Alg:CuttingModular}
Let $p/q$ be a slope, where $p$ and $q$ are both positive. Then for $j=1, \dots, n-1$:
\begin{enumerate}
\item If the $j$-th letter is $A$ and the next letter is $B$, add $L$.
\item If the $j$-th letter is $B$ followed by $A$, add $R$.
\item If the $j$-th letter is $A$ followed by $A$, add $RL$.
\item If the $j$-th letter is $B$ followed by $B$, add $LR$.
\end{enumerate}
If the slope is $0/1$ or $1/0$ (these are both lifts of the same modular geodesic) the cutting sequence is $LR$.
\end{algorithm}

See Figure~\ref{Fig:CuttingSequence}.
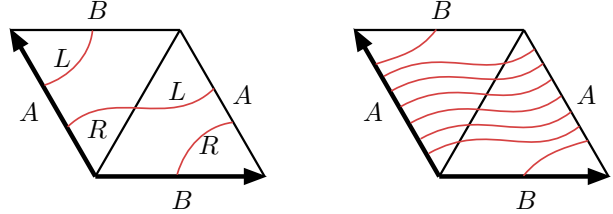
\begin{figure}
\centering
\begin{overpic}[width=4.5cm]{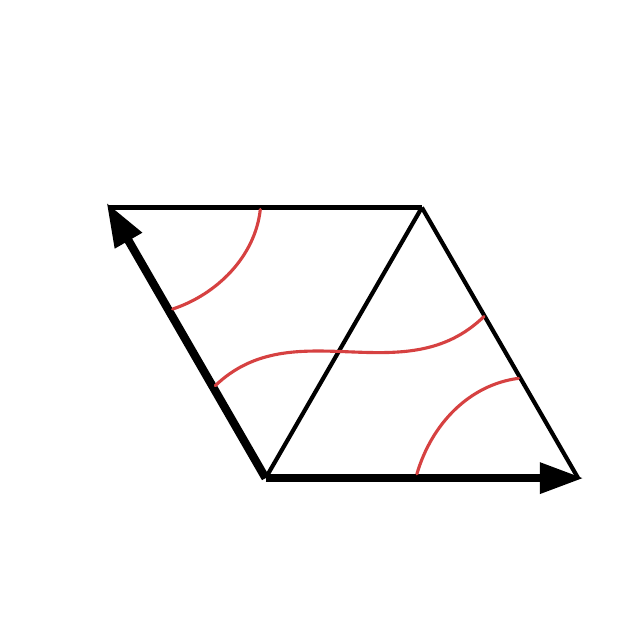}
    \put(20,40){$A$}
    \put(83,45){$A$}
    \put(65,14){$B$}
    \put(40,70){$B$}
    \put(73,30){$R$}
    \put(30,56){$L$}
    \put(40,35){$R$}
    \put(64,46){$L$}
\end{overpic}
\begin{overpic}[width=4.5cm]{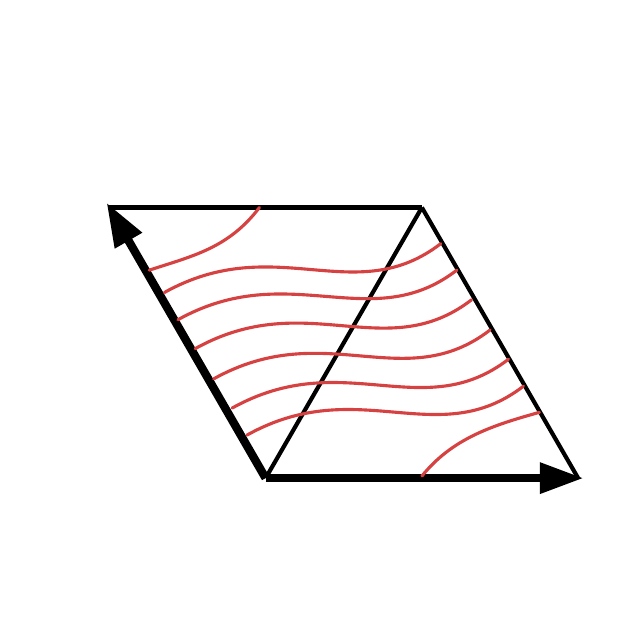}
    \put(20,40){$A$}
    \put(83,45){$A$}
    \put(65,14){$B$}
    \put(40,70){$B$}
\end{overpic}
\caption{On the left is the general rule for determining the cutting sequence of a positive slope, on the right is the cutting sequence $LR(RL)^6$ corresponding to the projection of the geodesic of slope $1/7$.}
\label{Fig:CuttingSequence}
\end{figure}

\begin{example}\label{Example:cuts}
Given a straight line of slope $1/n$, its $AB$-cutting sequence is the bi-infinite sequence given by concatenating copies of $BA^n$. Its $LR$-cutting sequence is the bi-infinite sequence given by concatenating $LR(RL)^{n-1}$. 
\end{example}

\subsection{Modular links}
Now return to the arithmetic modular links of Theorem~\ref{Thm:ModularArithmetic1}. We will construct examples of such links in the trefoil complement in the 3-sphere. 

From the proof of that theorem, the links are obtained by adding curves from the Farey triangulation that are invariant under the rotation $W$ rotating $0/1$ to $1/1$, $1/1$ to $1/0$, and $1/0$ to $0/1$. The smallest collection of curves comes from the initial triangle $0/1$, $1/1$, and $1/0$. All three curves at the vertices of this triangle are identified when we project to $\SigmaMod$. Hence we may use any of the three curves to determine the modular link. We take $p/q = 1/1$. 

Then observe that the $AB$-cutting sequence in this case is simply obtained by concatenating copies of $BA$. By Algorithm~\ref{Alg:CuttingModular}, the $LR$-cutting sequence is then obtained by concatenating copies of $LR$ (or equivalently $RL$). Therefore the modular geodesic corresponds to $RL$. In Figure~\ref{Fig:Parallelogram}, shown are the three distinct lifts of this geodesic in the parallelogram that is a fundamental domain for $\Sigma_{1,1}$. There are six lifts in total. As discussed above, the other three lifts traverse these curves in opposite directions. Note all six curves determine a cutting sequence $RL$ or $LR$, which gives the same bi-infinite sequence. 

\begin{figure}[ht!]
    \centering
    \begin{overpic}[width=4cm]{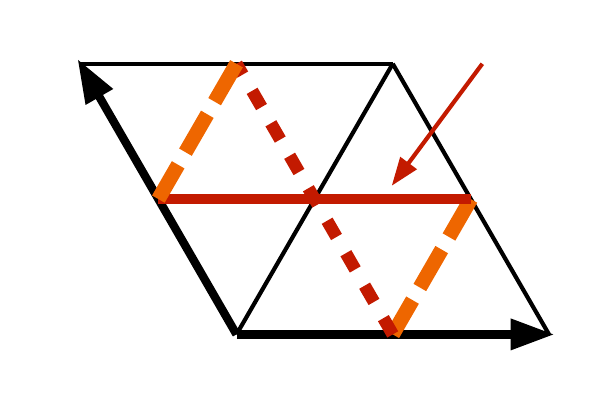}
    \put(95,8){$\lambda$}
    \put(6,55){$\mu$}
    \put(82,56){$RL$}
    \end{overpic}
    \caption{A fundamental domain for the two dimensional torus, and {three} different lifts corresponding to the modular geodesic $RL$.}
    \label{Fig:Parallelogram}
\end{figure}

Thus we have proved:
\begin{lemma}\label{Lem:Whitehead}
 The modular geodesic $RL$ is arithmetic.   \qed
\end{lemma}

The corresponding curve in the trefoil complement is obtained by drawing a closed curve on the branched surface of Figure~\ref{Fig:BranchedSurface}. The cutting sequence $LR$ instructs us that this curve must first run over the $L$ lobe of the branched surface, then the $R$ lobe, then close.
This is shown on the left of Figure~\ref{Fig:Whitehead}. Note that Lemma~\ref{Lem:Whitehead} is easily proved directly by the fact that its complement is homeomorphic to the Whitehead link complement as shown by the deformations of Figure~\ref{Fig:Whitehead}. 

\begin{figure}[h!]
    \centering
    \includegraphics[width=12cm]{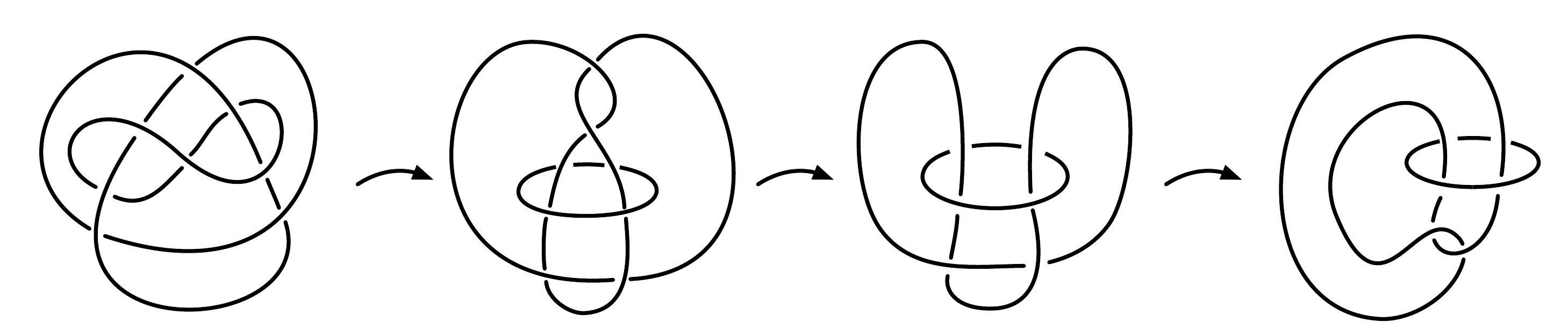}
    \caption{The homeomorphism between the complement of the $RL$ geodesic and the Whitehead link complement.}
    \label{Fig:Whitehead}
\end{figure}

Now consider the next simplest arithmetic modular link arising from the construction in the proof of Theorem~\ref{Thm:ModularArithmetic1}. This is obtained by adding a single additional curve, coming from a new vertex of a Farey triangle of distance one from that with vertices $1/0$, $1/1$, and $0/1$, and then taking the image of this curve under the degree three rotation. We see from Figure~\ref{Fig:Farey} that the only possibility is to next include $2/1$, $-2/1$, and $1/2$, which are all identified in $\SigmaMod$. 

In particular, the curve $1/2$ has $AB$-cutting sequence $BAA$, and $LR$-cutting sequence obtained by concatenating copies of $LRRL$, which is equivalent to $L^2 R^2$. Thus in the trefoil complement, it runs twice over the $L$ lobe of the branched surface, then twice over the $R$ lobe, before closing up.

The link given by the union of $LR$ and $L^2 R^2$ is also arithmetic, by Theorem~\ref{Thm:ModularArithmetic1}. It is shown on the left of Figure~\ref{Fig:ModularArithmeticExamples}.  This is a three component link in $S^3$. 
{As mentioned, any  finite union of modular geodesics has an embedding as orbits on the template. We remark this embedding is unique, and can be found in general using an algorithm, for example as in Birman and Williams}~\cite[Algorithm~2.4.3]{BirmanWilliams1983}; {see also Hui and Rodr\'iguez-Migueles}~\cite{HuiRodriguezMigueles}.

\begin{figure}
\centering
\includegraphics[width=3.5cm]{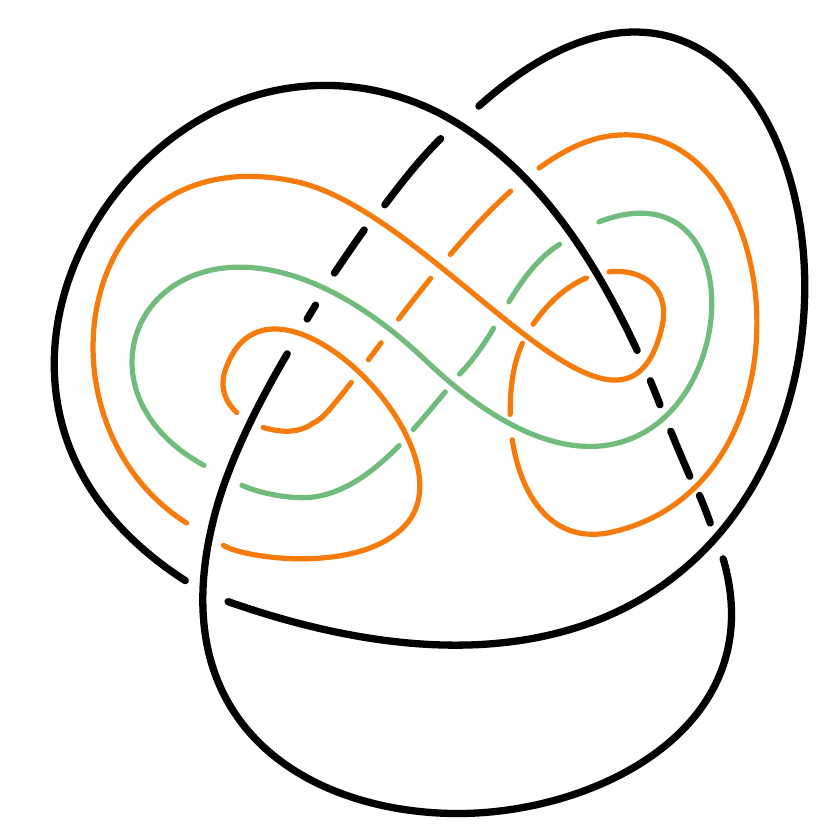}
\includegraphics[width=3.5cm]{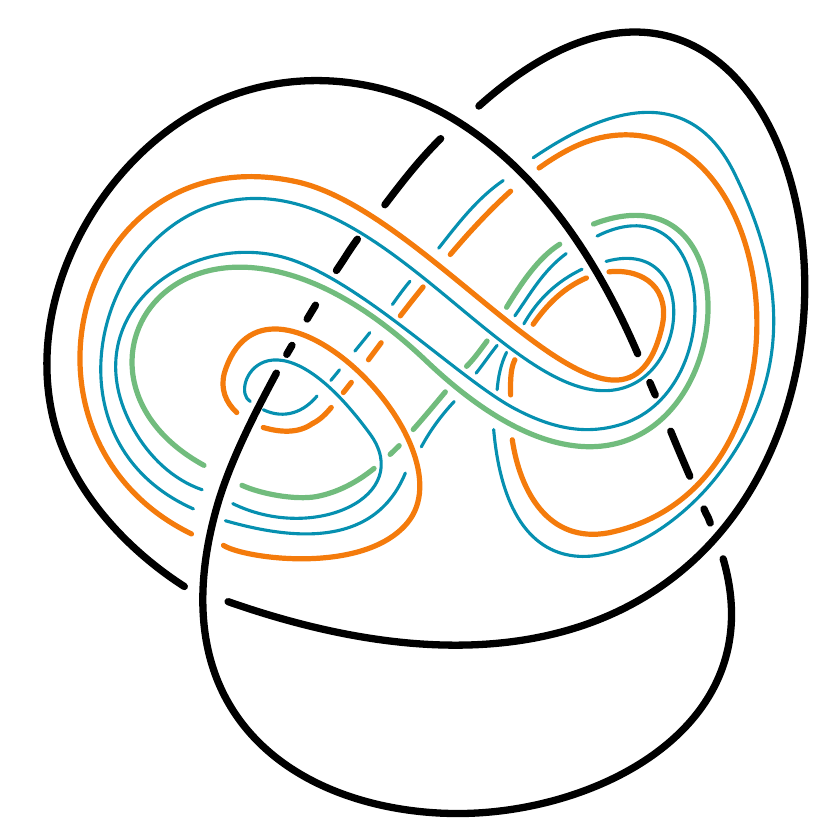}
\includegraphics[width=3.5cm]{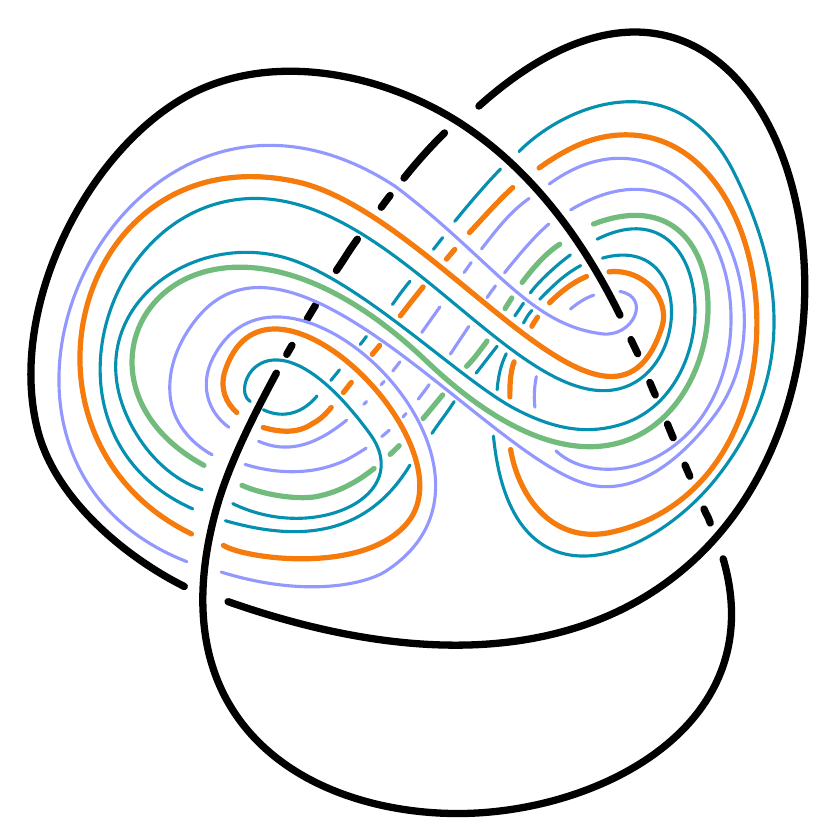}
\caption{After the Whitehead link, the next three simplest arithmetic links from Theorem~\ref{Thm:ModularArithmetic1} are shown {(note that in our conventions the symbol $R$ corresponds to the left side of the figure)}.}
\label{Fig:ModularArithmeticExamples}
\end{figure}

There are two choices for a four component link in $S^3$ that arises from Theorem~\ref{Thm:ModularArithmetic1}. 
One choice is to add slopes $3/2$, $-2/1$, and $1/3$, which are identified to a modular curve with $LR$-cutting sequence with repeating portion {$LRL^2R^2$}. Thus the four component arithmetic link in $S^3$ consists of the trefoil and the geodesics $LR$, $L^2R^2$, and {$LRL^2R^2$.} This link is shown in the middle of Figure~\ref{Fig:ModularArithmeticExamples}.

The other option is to add slopes $3/1$, $-1/2$, and $2/3$, which are identified to a modular curve with $LR$-cutting sequence with repeating portion {$LR^2L^2R$.} Thus another four component arithmetic link in $S^3$ consists of the trefoil, the link $LR$, $L^2R^2$, and {$LR^2L^2R$.}  

Note that the five component link consisting of the trefoil and the geodesics $LR$, $L^2R^2$, $LR^2L^2R$ and $LRL^2R^2$ is also arithmetic by Theorem~\ref{Thm:ModularArithmetic1}. This link is shown on the right of Figure~\ref{Fig:ModularArithmeticExamples}.

\section{Volume versus hyperbolic length}

In this section, our goal is to make explicit the relationship between volume of the canonical lift complement and geometric length of the original geodesic, for some sequence of geodesics in some surfaces.

\begin{remark}\label{Rmk:TraceLength}
Recall that for $A\in \PSL(2,\mathbb{R})$ a hyperbolic element of trace $t$, the eigenvalues of $A$ are $\frac{-t\pm\sqrt{t^2-4}}{2}$. Let $\lambda_A$ be the eigenvalue satisfying $|\lambda_A|>1$. Then the length of the closed geodesic determined by $A$ is $2\ln|\lambda_A|$.
\end{remark}

\begin{lemma}\label{Lem:length}
Let $\gamma_n$ be the unique closed geodesic on the modular surface lifting to the geodesic $1/n$ on $\Sigma_{1,1}$.
For $\Gamma_n:=\{{\gamma_i}\}_{i=1}^n$, the length $\ell(\Gamma_n)$ satisfies
\[
\ell(\Gamma_n)\asymp n^2.
\]
\end{lemma}

\begin{proof}
The matrix representative corresponding to $1/n$ is $A_n:=LR(RL)^{n-1}$; see Example~\ref{Example:cuts}. 
Let  
\[
\left(\begin{array}{cc}
	a_n&b_n\\
	c_n&d_n
	\end{array}\right):=(RL)^{n-1}, 
 \quad \mbox{so} \quad
\left(\begin{array}{cc}
	a_{n+1}&b_{n+1}\\
	c_{n+1}&d_{n+1}
	\end{array}\right)= \left(\begin{array}{cc}
	a_n+c_n & b_n+d_n\\
	a_n+2c_n & b_n+2d_n
\end{array}\right), 
\]
and
\[ 
A_{n+1}= \left(\begin{array}{cc}
	3a_n+4c_n & 3b_n+4d_n\\
	2a_n+3c_n & 2b_n+3d_n
	\end{array}\right).
\]
Then
$(3/2)\Trace{A_{n-1}}\leq \Trace{A_n}\leq 4\Trace{A_{n-1}}$. As $\Trace{A_{1}}=3$, by induction
\[\left(\frac{3}{2}\right)^n\leq \Trace{A_n}\leq 4^n.\]
The eigenvalue $\lambda_{n}$ of $A_n$ with $|\lambda_{n}|>1$ is bounded by
\[ \frac{|\lambda_{n}|}{2}\leq\frac{\Trace{A_n}}{2}\leq |\lambda_{n}|. \]
Thus the length of $\gamma_n$ satisfies
\[ n\ln\left(\frac{3}{2}\right) \leq \ell(\gamma_n)\leq 2n\ln(4), \]
and thus
\[n^2\ln\left(\frac{3}{2}\right)\leq  \ell(\Gamma_n)\leq  2n^2\ln(4).\qedhere \]
\end{proof}

\begin{corollary}\label{Cor:VolumeLengthTorus}
Let $\Gamma_k:=\left\{{\gamma_{1,n} = 1/n}, {\gamma_{2,n} = n/(n-1)}, {\gamma_{3,n} = (1-n)/1}\right\}_{n=1}^k$ be a collection of oriented simple closed geodesics on the once-punctured torus with a hyperbolic metric $\rho.$ Then for $\widehat{\Gamma_k}$ the canonical lifts of $\Gamma_k$,
\begin{enumerate}
\item $\UT(\Sigma_{1,1})\setminus{\widehat{{\Gamma}_k}}$ is arithmetic,
\item $\Vol(\UT(\Sigma_{1,1})\setminus{\widehat{{\Gamma}_k}})= 3k\,v_{oct}$, and
\item $\Vol(\UT(\Sigma_{1,1})\setminus{\widehat{{\Gamma}_k}})\asymp \sqrt{\ell_{\rho}(\Gamma_k)}$.
\end{enumerate}

\end{corollary}

\begin{proof}
Notice that $\Gamma_k$ are Farey neighbours, so by Theorem~\ref{Thm:CanonicalLiftFareyNeighbours}, $\UT(\Sigma_{1,1})\setminus{\widehat{{\Gamma}_k}}$ is arithmetic and
\[\Vol(\UT(\Sigma_{1,1})\setminus{\widehat{{\Gamma}_k}})=3k\,v_{oct}.\]
Observe that the geodesics $1/n$, $n/(n-1)$, $(1-n)/1$ project under the $6$-fold cover of the modular surface to $LR(RL)^{n-1}$; see Example~\ref{Example:cuts}. Then by Lemma~\ref{Lem:length}, the length of the projection of $\Gamma_k$ to $\SigmaMod$ is coarsely equivalent to $k^2$. Thus in the 6-fold cover $\Sigma_{1,1}$, the lengths satisfy
\[\ell_{\rho_{1,1}}({\Gamma_k})\asymp 6k^2,\]
where $\rho_{1,1}$ is the pullback metric induced on $\Sigma_{1,1}$ by the metric on the modular surface $\SigmaMod$. Then
\[ 
\Vol(\UT(\Sigma_{1,1})\setminus{\widehat{{\Gamma}_k}})\asymp v_{oct}\sqrt{\frac{3}{2}}\sqrt{\ell_{\rho_{1,1}}({\Gamma_k})}.
\]

The proof of this result for any hyperbolic metric on the once-punctured torus follows from the fact that any pair of hyperbolic metrics on a hyperbolic surface are bilipschitz; see for example \cite[Lemma~4.1]{BergeronPinskySilberman}. 
\end{proof}

By projecting the geodesics in Corollary~\ref{Cor:VolumeLengthTorus} under the $6$-fold cover to the modular surface we obtain the following result from the introduction. 

\begin{named}{Corollary~\ref{Cor:ModularVolumes}}
There exists a sequence $\{\gamma_k\}_{k\in\mathbb{N}}$ of closed geodesics on the modular surface with length $\ell({\gamma_k})\nearrow \infty$ such that for $\Gamma_n:=\bigcup_{k=1}^n \gamma_k$
\begin{enumerate}
  \item $\UT(\SigmaMod)\setminus{\widehat{\Gamma}_n}$ is arithmetic,
  \item $\Vol(\UT(\SigmaMod)\setminus{\widehat{{\Gamma}_n}})= n\,v_{oct}/2$, and
  \item $\Vol(\UT(\SigmaMod)\setminus{\widehat{{\Gamma}_n}})\asymp \sqrt{\ell({\Gamma_n})}$.
\end{enumerate}
Here $v_{oct}$ is the volume of a regular ideal octahedron. 
\end{named}

\begin{named}{Corollary~\ref{Cor:Cover}}
Let $\Sigma_{g,r}$ be an orientable punctured surface with any hyperbolic metric. Then there exists a sequence $\{{\Gamma_k}\}_{k\in\mathbb{N}}$ of filling finite sets of closed geodesics on $\Sigma_{g,r}$ with lengths $\ell({\Gamma_k})\nearrow \infty$, such that  $\UT(\Sigma_{g,r})\setminus{\widehat{{\Gamma_k}}}$ is arithmetic for each $k\in\mathbb{N}$ and
\[ 
Vol(\UT(\Sigma_{g,n})\setminus{\widehat{{\Gamma_k}}})\asymp \sqrt{\ell({\Gamma_k})}.
\]
\end{named}

\begin{proof}
By Remark~\ref{rem: punctured surfaces} we can construct a finite (branched) covering map $p$ from any orientable punctured hyperbolic surface $\Sigma_{g,r}$ of genus $g$ with $r$ punctures to the modular surface $\SigmaMod.$

Let $\widetilde \Gamma_k$ be the finite set of closed geodesics on $\Sigma$ obtained as the preimage under $p$ of the closed geodesics $\{\gamma_n\}_{n=1}^k$ of Lemma~\ref{Lem:length}). By Lemma~\ref{Lem:OctahedraArithmetic}, $\UT(\Sigma_{g,r})\setminus{\widehat{{\Gamma_k}}}$ is arithmetic. A similar estimation of the volume and lengths as in Corollary~\ref{Cor:VolumeLengthTorus} gives
\[
\Vol(\UT(\Sigma_{g,r})\setminus{\widehat{{\Gamma_k}}})\asymp \sqrt{\ell_{\rho}({\Gamma_k})},
\]
with the length $\ell_{\rho}(\Gamma_k)$ is measured in the pullback metric $\Sigma_{g,r}$ induced by the metric on $\SigmaMod$.
Again the proof of this result for any hyperbolic metric on $\Sigma_{g,r}$ follows from the fact that any pair of hyperbolic metrics on a hyperbolic surface are bilipschitz; for example \cite[Lemma~4.1]{BergeronPinskySilberman}.
\end{proof}

\section{Further questions}

There is only one arithmetic knot complement in the 3-sphere, namely the figure-8 knot, due to Reid~\cite{Reid:ArithmeticKnot}. 
{Is the modular geodesic $LR$ the only modular geodesic with arithmetic complement of its canonical lift? Notice that the question has a negative answer in the general context of any knot in the complement of the trefoil. Hatcher found an example of an arithmetic two-component link, which one component is the trefoil knot, and the trace field is $\QQ(\sqrt{-2})$}; see Figure~17 in~\cite{Hatcher:trefoilarithmeticlink}. However, the unknotted component in Hatcher's example is not a canonical lift of a closed geodesic in the modular surface.

All arithmetic modular links produced in this paper are conjugate in $\PSL(2,\CC)$ to a group commensurable with $\PSL(2,\ZZ(\sqrt{-1}))$. Are there examples of arithmetic modular links conjugate to groups commensurable with $\PSL(2, O_d)$ for $O_d$ a ring of integers in a different quadratic number field $\QQ(\sqrt{-d})$? More generally, is some classification possible? For example, in the 3-sphere, there are infinitely many arithmetic links. However, Baker and Reid showed that there are only finitely many principal congruence link complements in the 3-sphere \cite{BakerReid2014}, where a non-compact finite volume hyperbolic 3-manifold is principal congruence if it is isometric to $\HH^3/\Gamma(I)$ where
$\Gamma(I)=\mathrm{ker}\{\PSL(2,O_d) \to \PSL(2,O_d/I)\}$ for some ideal $I$ in $O_d$. Baker, Goerner, and Reid have now enumerated all principal congruence link complements in the 3-sphere~\cite{BakerGoernerReid}. Is a similar classification possible for modular links?

Any closed geodesic on the modular surface naturally corresponds to a real quadratic extension of $\mathbb{Q}$~\cite{SARNAK}. Does the arithmeticity of the complement of the corresponding canonical lift relate to this? For the examples in this paper, the quadratic field corresponding to the $LR$ geodesic is $\QQ(\sqrt{5})$. The geodesic $L^2R^2$ has quadratic field $\QQ(\sqrt{2})$. The geodesics $LR^2L^2L$ and $LRL^2R^2$ have the same length, and both correspond to the same quadratic field $\QQ(\sqrt{221})$. In general, geodesics corresponding to different maximal ideals in the same quadratic field will have the same length.

\bibliography{references}
\bibliographystyle{amsplain}

\end{document}